\documentclass[10pt]{amsart}%
\usepackage{amsfonts}
\usepackage{amsmath}
\usepackage{amssymb}
\usepackage{graphicx}%
\providecommand{\U}[1]{\protect\rule{.1in}{.1in}}
\newtheorem{theorem}{Theorem}
\theoremstyle{plain}

\numberwithin{equation}{section}
\begin{document}
\title[On the Bohnenblust-Hille and Littlewood's $4/3$ inequalities]{On the Bohnenblust-Hille inequality and a variant of Littlewood's $4/3$ inequality}
\author{Daniel Nu\~{n}ez-Alarc\'{o}n, Daniel Pellegrino, J.B. Seoane-Sep\'{u}lveda}
\address{Departamento de Matem\'{a}tica, UFPB, Jo\~{a}o Pessoa, PB, Brazil\\
Departamento de An\'{a}lisis Matem\'{a}tico,\\
\indent Facultad de Ciencias Matem\'{a}ticas, \\
\indent Plaza de Ciencias 3, \\
\indent Universidad Complutense de Madrid,\\
\indent Madrid, 28040, Spain.}
\email{jseoane@mat.ucm.es}
\subjclass[2010]{46G25, 47L22, 47H60.}
\keywords{Bohnenblust-Hille Theorem, Littlewood's $4/3$ inequality, Steinhaus random variables.}

\begin{abstract}
The search for sharp constants for inequalities of the type Littlewood's $4/3$
and Bohnenblust-Hille, besides its pure mathematical interest, has shown
unexpected applications in many different fields, such as Analytic Number
Theory, Quantum Information Theory, or (for instance) in deep results on the
$n$-dimensional Bohr radius. The recent estimates obtained for the multilinear
Bohnenblust-Hille inequality (in the case of real scalars) have been recently
used, as a crucial step, by A. Montanaro in order to solve problems in the
theory of quantum XOR games. Here, among other results, we obtain new upper
bounds for the Bohnenblust-Hille constants in the case of complex scalars. For
bilinear forms, we obtain the optimal constants of variants of Littlewood's
$4/3$ inequality (in the case of real scalars) when the exponent $4/3$ is
replaced by any $r\geq\frac{4}{3}.$ As a consequence of our estimates we show
that the optimal constants for the real case are always strictly greater than
the constants for the complex case.

\end{abstract}
\maketitle


\section{Introduction}

Let $\mathbb{K}$ stand for either $\mathbb{R}$ or $\mathbb{C}$. Littlewood's
$4/3$ inequality \cite{Litt} (see also \cite{garling}) asserts that there is a
constant $L_{\mathbb{K}}\geq1$ such that
\[
\left(  \sum\limits_{i,j=1}^{N}\left\vert U(e_{i},e_{j})\right\vert ^{\frac
{4}{3}}\right)  ^{\frac{3}{4}}\leq L_{\mathbb{K}}\left\Vert U\right\Vert
\]
for every bilinear form $U:\ell_{\infty}^{N}\times\ell_{\infty}^{N}%
\rightarrow\mathbb{K}$ and every positive integer $N.$ It is well known that
the exponent $4/3$ is optimal and it was recently shown in \cite{DD2} that the
constant $L_{\mathbb{R}}=\sqrt{2}$ is also optimal. For complex scalars we
just know that $L_{\mathbb{C}}\leq2/\sqrt{\pi}.$

However, if we replace $4/3$ by $r>4/3$, it is not difficult to prove that the
optimal constant $L_{\mathbb{K},r}$ satisfying
\begin{equation}
\left(  \sum\limits_{i,j=1}^{N}\left\vert U(e_{i},e_{j})\right\vert
^{r}\right)  ^{\frac{1}{r}}\leq L_{\mathbb{K},r}\left\Vert U\right\Vert
\label{p9}%
\end{equation}
is smaller than $\sqrt{2}$ (real case) and $2/\sqrt{\pi}$ (complex case). In
this note, among other results, we obtain the optimal constants $L_{\mathbb{R}%
,r}$ for all $r\geq\frac{4}{3};$ in fact, we prove that%
\[
L_{\mathbb{R},r}=\left\{
\begin{array}
[c]{c}%
2^{\frac{2-r}{r}}\text{ for }r\in\left[  \frac{4}{3},2\right)  \\
1\text{ for }r\geq2.
\end{array}
\right.
\]
As a consequence of our estimates we show that
\[
L_{\mathbb{R},r}>L_{\mathbb{C},r}%
\]
for all $r\in\left[  \frac{4}{3},2\right)  .$

Bohnenblust and Hille's inequality \cite{bh} is an improvement of Littlewood's
$4/3$ inequality, generalized to multilinear forms (see also \cite{annals2011,
defant2, defant} for recent approaches): for every positive integer $m$ there
is a constant $C_{m}\geq1$ so that
\[
\left(  \sum\limits_{i_{1},\ldots,i_{m}=1}^{N}\left\vert U(e_{i_{^{1}}}%
,\ldots,e_{i_{m}})\right\vert ^{\frac{2m}{m+1}}\right)  ^{\frac{m+1}{2m}}\leq
C_{m}\sup_{z_{1},...,z_{m}\in\mathbb{D}^{N}}\left\vert U(z_{1},...,z_{m}%
)\right\vert
\]
for every $m$-linear form $U:\ell_{\infty}^{N}\times\cdots\times\ell_{\infty
}^{N}\rightarrow\mathbb{C}$ and every positive integer $N$ (for polynomial
versions of the Bohnenblust-Hille inequality we refer to \cite{annals2011}).
The first upper estimate for $C_{m}$ is $m^{\frac{m+1}{2m}}2^{\frac{m-1}{2}}$,
which was further improved to $2^{\frac{m-1}{2}}$ in \cite{Ka}, to $\left(
\frac{2}{\sqrt{\pi}}\right)  ^{m-1}$ in \cite{Que} and, recently even better
constants, with optimal asymptotic behavior, were obtained in \cite{jmaaBH,
DD} (for related results see \cite{DD2, LAMA, APa}).

The original motivation of the Bohnenblust-Hille inequality rests on the famous Bohr's
absolute convergence problem, which consists in determining the maximal width
$T$ of the vertical strip in which a Dirichlet series ${\textstyle\sum
\limits_{n=1}^{\infty}}a_{n}n^{-s}$ converges uniformly but not absolutely.
The Bohnenblust-Hille inequality is a crucial tool to give a final solution to
Bohr's problem: $T=1/2$.

In Section 2 we improve the best known constants for the complex
Bohnenblust--Hille inequality. Besides the intrinsic mathematical interest of finding sharper constants for famous inequalities,
the search for better constants in Bohnenblust--Hille type inequalities has a
long history motivated by concrete goals. As an illustration we recall that,
in 2011, by proving that the polynomial Bohnenblust--Hille inequality is
hypercontractive, A. Defant, L. Frerick, J. Ortega-Cerd\'{a}, M. Ouna\"{\i}es
and K. Seip obtained, as consequence, several new results related to the study
of Dirichlet series. For instance, they obtain an ultimate generalization of a
result by H. P. Boas and D. Khavinson \cite{boas} on the $n$-dimensional Bohr radius. As we already mentioned in the \emph{Abstract},
one of the most recent applications of the Bohnenblust-Hille inequality
resides in the field of Quantum Information Theory, since the exact growth of
$C_{m}$ is related to a conjecture of Aaronson and Ambainis \cite{AA} about
classical simulations of quantum query algorithms (see, also, \cite{junge2010}%
). We also mention \cite{montanaro} for applications of the estimates from
\cite{jmaaBH} to Quantum Information Theory.


\section{The role of Steinhaus variables. Improving the constants in the
Bohnenblust--Hille inequality}

Let $\varepsilon_{1},\ldots,\varepsilon_{n}$ be a sequence of independent
random variables on some probability space $\left(  \Omega,\Sigma,P\right)  ,$
having uniform (with respect to the Lebesgue measure) distribution on the
complex unit-circle
\[
\left\{  z\in\mathbb{C}:\left\vert z\right\vert =1\right\}  .
\]
These are the so-called Steinhaus random variables. The usefulness of Steinhaus random variables in the proof of the
Bohnenblust--Hille inequality seems to have been first observed by H.
Queff\'{e}lec \cite{Que}. In our present approach we change the proof
presented in \cite{defant, jmaaBH} by replacing the usual Rademacher functions
by Steinhaus variables.

The first result allowing us to improve the constants of the Bohnenblust--Hille
inequality is a technical inequality (Theorem \ref{d}) which is a version (now
for of Steinhaus variables) of a similar result presented in \cite{defant,
jmaaBH} for Rademacher functions. The crucial point in our argument is that
the constants which arise in Theorem \ref{d} are derived from the constants
that appear in the Khinchine inequality for Steinhaus variables and, as we
shall see at the end of the paper, this procedure generates sharper constants
for the Bohnenblust--Hille inequality.

Let us recall Khinchine's inequality (for Steinhaus variables) and other
useful result:

\begin{theorem}
[Khinchine's inequality]\label{kkk4} For every $0<p<\infty$, there exist
constants $\widetilde{A_{p}}$ and $\widetilde{B_{p}}$ such that
\begin{equation}
\widetilde{A_{p}}\left(  \sum_{n=1}^{N}\left\vert a_{n}\right\vert
^{2}\right)  ^{\frac{1}{2}}\leq\left\Vert \sum_{n=1}^{N}a_{n}\varepsilon
_{n}\right\Vert _{p}\leq\widetilde{B_{p}}\left(  \sum_{n=1}^{N}\left\vert
a_{n}\right\vert ^{2}\right)  ^{\frac{1}{2}} \label{lpo4}%
\end{equation}
for every positive integer $N$ and scalars $a_{1},\ldots,a_{N}.$
\end{theorem}

From \cite[p. 151]{pel} we know that
\begin{equation}
A_{p}\leq\widetilde{A_{p}} \label{ppee}%
\end{equation}
for all $p$ (here $A_{p}$ denotes the constants that appear in the place of
$\widetilde{A_{p}}$ in Khinchine's inequality for Rademacher functions). For
example, when $p=1$ it is well known that $A_{p}=\frac{1}{\sqrt{2}}%
\approx0.707$ and $\widetilde{A_{p}}=\frac{\sqrt{\pi}}{2}\approx
\allowbreak0.886.$

For details on the Khinchine inequalities we refer to \cite[Theorem 1.10]{Di}
for the case of Rademacher functions and to \cite[Section 2]{ba} for more
general cases, including the case of Steinhaus variables.

The following result, crucial for the proof of the Bohnenblust--Hille
inequality, has essentially the same proof of its analogous for Rademacher
functions (see \cite{defant, jmaaBH}).

\begin{theorem}
\label{d} Let $1\leq r\leq2$, and let $(y_{i_{1},\ldots,i_{m}})_{i_{1}%
,\ldots,i_{m}=1}^{N}$ be a matrix in $\mathbb{C}$. Then
\[
\left(  \sum\limits_{i_{1},\ldots,i_{m}=1}^{N}\left\vert y_{i_{1},\ldots
,i_{m}}\right\vert ^{2}\right)  ^{1/2}\leq\left(  \widetilde{A_{r}}\right)
^{-m}\left\Vert \sum\limits_{i_{1},\ldots,i_{m}=1}^{N}\varepsilon_{i_{1}%
}...\varepsilon_{i_{m}}y_{i_{1}\ldots i_{m}}\right\Vert _{r}.
\]

\end{theorem}

In view of (\ref{ppee}) we conclude that the constants $\left(  \widetilde
{A_{r}}\right)  ^{-m}$ are not greater than the constants from its analogous
for Rademacher functions and for this reason we shall have better estimates
for the constants in the Bohnenblust--Hille inequality.

The proof of the Bohnenblust--Hille inequality is (replacing the Rademacher
functions for Steinhaus variables) the same proof as that from \cite{jmaaBH}.
The difference in the constants is a consequence from the new constants from
the Khinchine inequality for Steinhaus variables.

\begin{theorem}
\label{aaa} If $m\geq1$, then
\[
\left(  \sum\limits_{i_{1},\ldots,i_{m}=1}^{N}\left\vert U(e_{i_{^{1}}}%
,\ldots,e_{i_{m}})\right\vert ^{\frac{2m}{m+1}}\right)  ^{\frac{m+1}{2m}}\leq
C_{m}\sup_{z_{1},...,z_{m}\in\mathbb{D}^{N}}\left\vert U(z_{1},...,z_{m}%
)\right\vert
\]
for every $m$-linear form $U:\ell_{\infty}^{N}\times\cdots\times\ell_{\infty
}^{N}\rightarrow\mathbb{C}$ and every positive integer $N$, with
\[
C_{1}=1,
\]
\[
C_{m}=\frac{C_{m/2}}{\left(  \widetilde{A_{\frac{2m}{m+2}}}\right)  ^{m/2}}%
\]
for $m$ even and
\[
C_{m}=\left(  \frac{C_{\frac{m-1}{2}}}{\left(  \widetilde{A_{\frac{2m-2}{m+1}%
}}\right)  ^{\frac{m+1}{2}}}\right)  ^{\frac{m-1}{2m}}\left(  \frac
{C_{\frac{m+1}{2}}}{\left(  \widetilde{A_{\frac{2m+2}{m+3}}}\right)
^{\frac{m-1}{2}}}\right)  ^{\frac{m+1}{2m}}.
\]
for $m$ odd.
\end{theorem}

As mentioned before, from \cite{pel} we know that $A_{p}\leq\widetilde{A_{p}}
$ for all $p$ and we easily conclude that the constants from Theorem \ref{aaa}
are better than the constants from the similar result from \cite{jmaaBH}. For
$p=1$, J. Sawa \cite{sawa} has shown that the best value for $\widetilde
{A_{p}}$ is
\[
\widetilde{A_{1}}=\frac{\sqrt{\pi}}{2}.
\]
Since $C_{1}=1$, we have
\[
C_{2}=\frac{2}{\sqrt{\pi}},
\]
as obtained previously by Queff\'{e}lec \cite{Que}. The evaluation of the precise values
for $C_{m}$ rests on the evaluation of precise values for $\widetilde{A_{p}}$
with
\[
p\in\left\{  \frac{2m}{m+2}:m\geq2\right\}  \cup\left\{  \frac{2m-2}%
{m+1}:m\geq3\right\}  \cup\left\{  \frac{2m+2}{m+3}:m\geq3\right\}
\subset\lbrack1,2).
\]

As an \textquotedblleft\textit{Added in proof} \textquotedblright\ in the same
paper \cite{sawa}, J. Sawa asserts that the sharpest constants for the
parameter $p$, with $p_{0}<p<2$ and $p_{0}\in\left(  0,2\right)  $ defined as
the unique root of the equation
\[
2^{p/2} \cdot\Gamma\left(  \frac{p+1}{2}\right)  =\sqrt{\pi}\left(
\Gamma\left(  \frac{p+2}{2}\right)  \right)  ^{2}%
\]
are
\begin{equation}
\widetilde{A_{p}}=\left(  \Gamma\left(  \frac{p+2}{2}\right)  \right)
^{\frac{1}{p}}. \label{hhh}%
\end{equation}
A $4$-digit approximation provides $p_{0}\approx0.4756.$ However, Sawa
presented no proof for his claim. But, fortunately, for $p\geq1$ H. K\"{o}nig
proved that \eqref{hhh} is, in fact, the precise value of $\widetilde{A_{p}}$
(see \cite[Section 2]{ba} and references therein). Using these values for
$\widetilde{A_{p}}$ we construct the following table, where one can check the
different estimates for $C_{m}$ that have been obtained so far.

\bigskip%

\begin{tabular}
[c]{c|c|c|c|c|c}%
$m$ & new constants & {\tiny \cite{jmaaBH} (2012)} & $\left(  \frac{2}
{\sqrt{\pi}}\right)  ^{m-1}${\tiny (\cite{Que, defant2}, 1995)} &
$2^{\frac{m-1}{2}}${\tiny (\cite{Ka},1978)} & $m^{\frac{m+1}{2m}}2^{\frac
{m-1}{2}}${\tiny (\cite{bh},1931)}\\\hline
$2$ & $\approx1.1284$ & $-$ & $\approx1.1284$ & $\approx1.414$ &
$\approx2.378$\\\hline
$3$ & $\approx1.2364$ & $-$ & $\approx1.273$ & $2$ & $\approx4.160$\\\hline
$4$ & $\approx1.3155$ & $-$ & $\approx1.437$ & $\approx2.828$ & $\approx
6.726$\\\hline
$5$ & $\approx1.3982$ & $-$ & $\approx1.621$ & $4$ & $\approx10.506$\\\hline
$6$ & $\approx1.4637$ & $-$ & $\approx1.829$ & $\approx5.657$ & $\approx
16.088$\\\hline
$7$ & $\approx1.5224$ & $\approx1.929$ & $\approx2.064$ & $8$ & $\approx
24.322$\\\hline
$8$ & $\approx1.5714$ & $\approx2.031$ & $\approx2.329$ & $\approx11.313$ &
$\approx36.442$\\\hline
$9$ & $\approx1.6298$ & $\approx2.172$ & $\approx2.628$ & $16$ &
$\approx54.232$\\\hline
$10$ & $\approx1.6800$ & $\approx2.292$ & $\approx2.965$ & $\approx22.627$ &
$\approx80.283$\\\hline
$11$ & $\approx1.7256$ & $\approx2.449$ & $\approx3.346$ & $32$ &
$\approx118.354$\\\hline
$12$ & $\approx1.7659$ & $\approx2.587$ & $\approx3.775$ & $\approx45.425$ &
$\approx173.869$\\\hline
$13$ & $\approx1.8061$ & $\approx2.662$ & $\approx4.260$ & $64$ &
$\approx254.680$\\\hline
$14$ & $\approx1.8422$ & $\approx2.728$ & $\approx4.807$ & $\approx90.509$ &
$\approx372.128$\\\hline
$15$ & $\approx1.8757$ & $\approx2.805$ & $\approx5.425$ & $128$ &
$\approx542.574$\\\hline
$16$ & $\approx1.9060$ & $\approx2.873$ & $\approx6.121$ & $\approx181.019$ &
$\approx789.612$\\\hline
$100$ & $\approx3.2968$ & $\approx7.603$ & $\approx1.55973\cdot10^{3}$ &
$\approx7.96131459\cdot10^{14}$ & $\approx8.14675743\cdot10^{15}$%
\end{tabular}


\subsection{Remarks on the optimal constants satisfying the complex
Bohnenblust--Hille inequality}

\label{acx}

Let $\left(  K_{n}\right)  _{n=1}^{\infty}$ be the sequence of the best
constants satisfying the complex Bohnenblust--Hille inequality. In
\cite{AlarconPellegrino} it was recently shown that $\left(  K_{n}\right)
_{n=1}^{\infty}$ does not have a polynomial growth and, besides, if
\[
K_{n}\sim n^{q},
\]
then%
\[
0\leq q\leq\log_{2}\left(  \frac{e^{1-\frac{1}{2}\gamma}}{\sqrt{2}}\right)
\approx0.52632,
\]
where $\gamma$ denotes the famous Euler-Mascheroni constant
\[
\gamma= \lim_{m \rightarrow\infty} \left(  \sum_{k=1}^{m} \frac{1}{k} - \log
m\right)  \approx0.57721.
\]

Since $\widetilde{A_{p}}=\left(  \Gamma\left(  \frac{p+2}{2}\right)  \right)
^{\frac{1}{p}}$, we have%
\[
\widetilde{A_{\frac{2m}{m+2}}}=\left(  \Gamma\left(  \frac{\frac{2m}{m+2}%
+2}{2}\right)  \right)  ^{\frac{m+2}{2m}}%
\]
and
\[
\frac{C_{m}}{C_{m/2}}=\left(  \Gamma\left(  \frac{\frac{2m}{m+2}+2}{2}\right)
\right)  ^{\frac{-m-2}{4}}.
\]
Using some basic properties of the Gamma function we can prove that
\[
\lim_{m\rightarrow\infty}\frac{C_{m}}{C_{m/2}}=\lim_{m\rightarrow\infty
}\left(  \Gamma\left(  \frac{\frac{2m}{m+2}+2}{2}\right)  \right)
^{\frac{-m-2}{4}}=e^{\frac{1}{2}-\frac{1}{2}\gamma}\approx1.23539,
\]
and following the arguments from the Dichotomy Theorem (see
\cite{AlarconPellegrino}) we conclude that if $K_{n}\sim n^{q}$, then
\[
0<q\leq\log_{2}\left(  e^{\frac{1}{2}-\frac{1}{2}\gamma}\right)
\approx0.30497,
\]
as we wished. We mention that a more complete results on this line were
recently proved in \cite{APa}.

\section{The variants of Littlewood's $4/3$ inequality}

\subsection{Real case}

As mentioned in the Introduction, if we replace $4/3$ by $r>4/3$, then the
optimal constant $L_{\mathbb{K},r}$ satisfying
\begin{equation}
\left(  \sum\limits_{i,j=1}^{N}\left\vert U(e_{i},e_{j})\right\vert
^{r}\right)  ^{\frac{1}{r}}\leq L_{\mathbb{K},r}\left\Vert U\right\Vert
\end{equation}
is smaller than $\sqrt{2}.$ Our main goal is to find the optimal values of
$L_{\mathbb{K},r}$ for all $r\geq4/3:$

\begin{theorem}
\label{2_2} The optimal constant $L_{\mathbb{R},r}$ satisfying (\ref{p9}) is
\[
L_{\mathbb{R},r}=\left\{
\begin{array}
[c]{c}%
2^{\frac{2-r}{r}}\text{ for }r\in\left[  \frac{4}{3},2\right) \\
1\text{ for }r\geq2.
\end{array}
\right.
\]

\end{theorem}

\begin{proof}
The case $r\geq2$ is quite simple. In fact, one can use that the real scalar
field has cotype $2$ and its cotype constant is $1.$ Hence%
\begin{equation}
\left(  \sum\limits_{i,j=1}^{N}\left\vert U(e_{i},e_{j})\right\vert
^{2}\right)  ^{\frac{1}{2}}\leq\left\Vert U\right\Vert \label{aaaq}%
\end{equation}
for all $N$ and all bilinear forms $U:\ell_{\infty}^{N}\times\ell_{\infty}%
^{N}\rightarrow\mathbb{R}$. Using (\ref{aaaq}) and the monotonicity of the
$\ell_{r}$ norms we conclude that $L_{\mathbb{R},r}\leq1.$ On the other hand,
using $U_{0}(x,y)=x_{1}y_{1}$ in (\ref{p9}) we conclude that $L_{\mathbb{R}%
,r}\geq1$. Now we deal with the case $r\in\lbrack\frac{4}{3},2)$. Using a
simple interpolation argument we can show that if $\theta\in\left(
0,1\right)  $ is so that%
\[
\frac{1}{r}=\frac{\theta}{4/3}+\frac{1-\theta}{2},
\]
then%
\begin{align*}
\left(  \sum\limits_{i,j=1}^{N}\left\vert U(e_{i},e_{j})\right\vert
^{r}\right)  ^{\frac{1}{r}} &  \leq\left(  \left(  \sum\limits_{i,j=1}%
^{N}\left\vert U(e_{i},e_{j})\right\vert ^{4/3}\right)  ^{\frac{3}{4}}\right)
^{\theta}\left(  \left(  \sum\limits_{i,j=1}^{N}\left\vert U(e_{i}%
,e_{j})\right\vert ^{2}\right)  ^{\frac{1}{2}}\right)  ^{1-\theta}\\
&  \leq\left(  \sqrt{2}\right)  ^{\theta}\left\Vert U\right\Vert \\
&  =2^{\frac{2-r}{r}}\left\Vert U\right\Vert .
\end{align*}
On the other hand, by considering%
\begin{equation}
U_{1}(x,y)=x_{1}y_{1}+x_{1}y_{2}+x_{2}y_{1}-x_{2}y_{2}\label{in}%
\end{equation}
we have $\left\Vert U_{1}\right\Vert =2$ and thus%
\[
L_{\mathbb{R},r}\geq\frac{4^{\frac{1}{r}}}{\left\Vert U_{1}\right\Vert
}=2^{\frac{2-r}{r}}.
\]

\end{proof}

\subsection{Complex case\label{imn}}

We will show that%
\[
L_{\mathbb{C},r}\leq\left(  \frac{2}{\sqrt{\pi}}\right)  ^{\frac{4-2r}{r}%
}\text{ for all }r\in\left[  \frac{4}{3},2\right)
\]
and for $r\geq2$ it is straightforward that $L_{\mathbb{C},r}=1.$ However, we
do not prove that our estimates are optimal for $r\in\left[  \frac{4}%
{3},2\right)  .$

For the case of complex scalars the more accurate known estimate for the
constant in the Littlewood's $4/3$ theorem is
\[
L_{\mathbb{C},\frac{4}{3}}\leq\frac{2}{\sqrt{\pi}}.
\]

The same interpolation argument used in the case of real scalars can be used to show that if $\theta\in\left(
0,1\right)  $ is so that%
\[
\frac{1}{r}=\frac{\theta}{4/3}+\frac{1-\theta}{2},
\]
then%
\begin{align*}
\left(  \sum\limits_{i,j=1}^{N}\left\vert U(e_{i},e_{j})\right\vert
^{r}\right)  ^{\frac{1}{r}}  &  \leq\left(  \left(  \sum\limits_{i,j=1}%
^{N}\left\vert U(e_{i},e_{j})\right\vert ^{4/3}\right)  ^{\frac{3}{4}}\right)
^{\theta}\left(  \left(  \sum\limits_{i,j=1}^{N}\left\vert U(e_{i}%
,e_{j})\right\vert ^{2}\right)  ^{\frac{1}{2}}\right)  ^{1-\theta}\\
&  \leq\left(  \frac{2}{\sqrt{\pi}}\right)  ^{\theta}\left\Vert U\right\Vert
\\
&  =\left(  \frac{2}{\sqrt{\pi}}\right)  ^{\frac{4-2r}{r}}\left\Vert
U\right\Vert .
\end{align*}
For $r\geq2$, it is simple to show that the exact value is $L_{\mathbb{C}%
,r}=1.$ However our technique to provide lower estimates for $L_{\mathbb{R},r}$
seems useless for the complex case. The following table is illustrative:

\begin{center}%
\begin{tabular}
[c]{c|c|c}%
$r$ & $L_{\mathbb{C},r}\geq$ & $L_{\mathbb{C},r}\leq$\\\hline
& $1$ & $\left(  \frac{2}{\sqrt{\pi}}\right)  ^{\frac{4-2r}{r}}$\\\hline
$4/3$ & $1$ & $2/\sqrt{\pi}\approx1.128380$\\\hline
$1.93$ & $1$ & $1.0088$\\\hline
$1.95$ & $1$ & $1.0062$\\\hline
$1.99$ & $1$ & $1.0012$\\\hline
$\geq2$ & $1$ & $1$%
\end{tabular}

\end{center}

Since $L_{\mathbb{R},r}=2^{\frac{2-r}{r}}$ and $L_{\mathbb{C},r}\leq\left(
\frac{2}{\sqrt{\pi}}\right)  ^{\frac{4-2r}{r}}$ for all $r\in\left[  \frac
{4}{3},2\right)  $, it follows that%
\begin{equation}
L_{\mathbb{R},r}>L_{\mathbb{C},r} \label{tt}%
\end{equation}
for all nontrivial cases, i.e., whenever $r\in\left[  \frac{4}{3},2\right)  .$

\subsection{Some remarks}

We do not know if our estimates for complex scalars are optimal. We now stress
that a different technique, although quite effective for estimates of the
Bohnenblust--Hille inequality, provides worse results. This approach is based
on recent arguments from \cite{defant, DD2, jmaaBH}). For the sake of
completeness, let us recall two useful results:

\begin{theorem}
[Khinchine's inequality]\label{k} For all $0<p<\infty$, there exist constants
$A_{p}$ and $B_{p}$ such that
\begin{equation}
A_{p}\left(  \sum_{n=1}^{N}\left\vert a_{n}\right\vert ^{2}\right)  ^{\frac
{1}{2}}\leq\left(  \int_{0}^{1}\left\vert \sum_{n=1}^{N}a_{n}r_{n}\left(
t\right)  \right\vert ^{p}dt\right)  ^{\frac{1}{p}}\leq B_{p}\left(
\sum_{n=1}^{N}\left\vert a_{n}\right\vert ^{2}\right)  ^{\frac{1}{2}}
\label{lpo}%
\end{equation}
for every positive integer $N$ and scalars $a_{1},\ldots,a_{n}$ ($r_{n}$
denotes the $n$-$th$ Rademacher function)$.$
\end{theorem}

For $p>p_{0}$ with $1<p_{0}<2$ defined by
\[
\Gamma\left(  \frac{p_{0}+1}{2}\right)  =\frac{\sqrt{\pi}}{2},
\]
a result due to U. Haagerup (\cite{haag}) asserts that
\begin{equation}
A_{p}:=\sqrt{2}\left(  \frac{\Gamma((p+1)/2)}{\sqrt{\pi}}\right)  ^{1/p}
\label{kkklll}%
\end{equation}
are the best constants satisfying (\ref{lpo}); for $p\leq p_{0}$ the best
values are
\begin{equation}
A_{p}=2^{\frac{1}{2}-\frac{1}{p}}. \label{0op}%
\end{equation}

\begin{theorem}
[{Blei, Defant et al., \cite[Lemma 3.1]{defant}}]\label{b} Let $A$ and $B$ be
two finite non-void index sets, and $(a_{ij})_{(i,j)\in A\times B}$ a scalar
matrix with positive entries, and denote its columns by $\alpha_{j}%
=(a_{ij})_{i\in A}$ and its rows by $\beta_{i}=(a_{ij})_{j\in B}.$ Then, for
$q,s_{1},s_{2}\geq1$ with $q>\max(s_{1},s_{2})$ we have
\[
\left(  \sum_{(i,j)\in A\times B}a_{ij}^{w(s_{1},s_{2})}\right)  ^{\frac
{1}{w(s_{1},s_{2})}}\leq\left(  \sum_{i\in A}\left\Vert \beta_{i}\right\Vert
_{q}^{s_{1}}\right)  ^{\frac{f(s_{1},s_{2})}{s_{1}}}\left(  \sum_{j\in
B}\left\Vert \alpha_{j}\right\Vert _{q}^{s_{2}}\right)  ^{\frac{f(s_{2}%
,s_{1})}{s_{2}}},
\]
with
\begin{align*}
w  &  :[1,q)^{2}\rightarrow\lbrack0,\infty),\text{ }w(x,y):=\frac
{q^{2}(x+y)-2qxy}{q^{2}-xy},\\
f  &  :[1,q)^{2}\rightarrow\lbrack0,\infty),\text{ }f(x,y):=\frac{q^{2}%
x-qxy}{q^{2}(x+y)-2qxy}.
\end{align*}

\end{theorem}

As we already know, the Khinchine inequality for Steinhaus variables has
\[
\widetilde{A_{p}}=\left(  \Gamma\left(  \frac{p+2}{2}\right)  \right)
^{\frac{1}{p}}%
\]
as the optimal constant whenever $p\geq1$ (see \cite{ba, sawa}). So, using
this value for $\widetilde{A_{p}}$ and an argument similar to the proof of the
main result of \cite{jmaaBH} (which has its roots in \cite{defant}) with%
\[
\left\{
\begin{array}
[c]{c}%
s_{1}=s_{2}=\frac{2r}{4-r},\\
q=2
\end{array}
\right.
\]
in Theorem \ref{b}, we have%
\[
\left\{
\begin{array}
[c]{c}%
w(s_{1},s_{2})=r,\\
f(s_{1},s_{2})=1/2.
\end{array}
\right.
\]

and we obtain
\[
L_{\mathbb{C},r}\leq\left(  \left(  \Gamma\left(  \frac{\frac{2r}{4-r}+2}%
{2}\right)  \right)  ^{\frac{4-r}{2r}}\right)  ^{-1}=\left(  \Gamma\left(
\frac{4}{4-r}\right)  \right)  ^{\frac{r-4}{2r}}%
\]
for all $r\in\left[  \frac{4}{3},2\right)  $. But a direct inspection shows
that%
\[
\left(  \Gamma\left(  \frac{4}{4-r}\right)  \right)  ^{\frac{r-4}{2r}}>\left(
\frac{2}{\sqrt{\pi}}\right)  ^{\frac{4-2r}{r}}%
\]
for all $r\in\left(  \frac{4}{3},2\right)  $ and thus the estimates of
Subsection \ref{imn} are more precise.


\begin{thebibliography}{99}                                                                                               %


\bibitem {AA}S. Aaronson, A. Ambainis, The Need for Structure in Quantum
Speedups, In Proceedings of ICS 2011, 338--352. arXiv:0911.0996

\bibitem {ba}A. Baernstein, R.C. Culverhouse, Majorization of sequences, sharp
vector Khinchin inequalities, and bisubharmonic functions, Studia Math. 152
(2002), 231--248.

\bibitem {boas}H.P. Boas, D. Khavinson, Bohr's power series theorem in several
variables, Proc. Amer. Math. Soc. 125 (1997), 2975--2979.

\bibitem {bh}H.F. Bohnenblust, E. Hille, On the absolute convergence of
Dirichlet series, Ann. of Math. (2) 32 (1931), 600--622.

\bibitem {Davie}A.M. Davie, Quotient algebras of uniform algebras, J. London
Math. Soc. (2) 7 (1973), 31--40.

\bibitem {annals2011}A. Defant, L. Frerick, J. Ortega-Cerd%
$\backslash$%
`\{a\}, M. Ounaies, K. Seip, The Bohnenblust--Hille inequality for homogeneous
polynomials is hypercontractive, Ann. of Math. (2), 174 (2011), 485--497.

\bibitem {defant}A. Defant, D. Popa, U. Schwarting, Coordinatewise multiple
summing operators in Banach spaces, J. Funct. Anal. 259 (2010), 220--242.

\bibitem {defant2}A. Defant, P. Sevilla-Peris, A new multilinear insight on
Littlewood's 4/3-inequality, J. Funct. Anal. 256 (2009),1642--1664.

\bibitem {Di}J. Diestel, H. Jarchow, A. Tonge, Absolutely summing operators,
Cambridge Studies in Advanced Mathematics 43, Cambridge University Press,
Cambridge, 1995.

\bibitem {DD}D. Diniz, \ G.A. Mu\~{n}oz-Fern\'{a}ndez, D. Pellegrino, J.B.
Seoane-Sep\'{u}lveda, The asymptotic growth of the constants in the
Bohnenblust--Hille inequality is optimal, J. Funct. Anal. 263 (2012), 415--428.

\bibitem {DD2}D. Diniz, \ G.A. Mu\~{n}oz-Fern\'{a}ndez, D. Pellegrino, J.B.
Seoane-Sep\'{u}lveda, Lower bounds for the constants in the Bohnenblust--Hille
inequality: the case of real scalars, Proc. Amer. Math. Soc., in press.

\bibitem {garling}D.J.H. Garling, Inequalities: a journey into linear
analysis, Cambridge University Press, Cambridge, 2007.

\bibitem {haag}U. Haagerup, The best constants in the Khintchine inequality,
Studia Math. 70 (1981), 231--283.

\bibitem {junge2010}M. Junge, C. Palazuelos, D. P\'{e}rez-Garc{\'{\i}}a, I. Villanueva,
M.M. Wolf, Unbounded violations of bipartite Bell inequalities via operator
space theory, Comm. Math. Phys. 300 (2010), 715--739.

\bibitem {Ka}S. Kaijser, Some results in the metric theory of tensor products,
Studia Math. 63 (1978), 157--170.

\bibitem {Litt}J.E. Littlewood, On bounded bilinear forms in an infinite
number of variables, Q. J. Math. 1 (1930), 164--174

\bibitem {LAMA}G.A. Munoz-Fernandez, D.\ Pellegrino, J.B. Seoane-Sep\'{u}lveda,
Estimates for the asymptotic behavior of the constants in the
Bohnenblust--Hille inequality, Linear Multilinear Algebra 60 (2012), 573--582.

\bibitem {montanaro}A. Montanaro, Some applications of hypercontractive
inequalities in quantum information theory, arXiv:1208.0161v2 [quant-ph].

\bibitem {AlarconPellegrino}D. Nu\~{n}ez-Alarc\'{o}n, D. Pellegrino, On the growth of
the optimal constants of the multilinear Bohnenblust--Hille inequality,
arXiv:1205.2385v1 [math.FA].

\bibitem {APa}D. Nu\~{n}ez-Alarc\'{o}n, D. Pellegrino, J.B. Seoane-Sep\'{u}lveda, D,M.
Serrano, Rodriguez, There exist multilinear Bohnenblust--Hille constants
$\left(  C_{n}\right)  _{n=1}^{\infty}$ with $\lim_{n}C_{n+1}-C_{n}=0$,
arXiv:1207.0124 [math.FA].

\bibitem {pel}A. Pe\l czy\'{n}ski, Norms of classical operators in function spaces,
Colloquium in honor of Laurent Schwartz, Vol. 1 (Palaiseau), Asterisque 131
(1985), 137--162.

\bibitem {jmaaBH}D. Pellegrino, J.B. Seoane-Sep\'{u}lveda, New upper bounds for
the constants in the Bohnenblust--Hille inequality, J. Math. Anal. Appl. 386
(2012), 300--307.

\bibitem {Que}H. Queff\'{e}lec, H. Bohr's vision of ordinary Dirichlet series: old
and new results, J. Anal. 3 (1995), 43--60.

\bibitem {sawa}J. Sawa, The best constant in the Khinchine inequality for
complex Steinhaus variables, the case $p=1$, Studia Math. 81 (1995), 107--126.
\end{thebibliography}
\end{document}